\documentclass[a4paper, 11pt]{article}

\usepackage{calc, amsmath, amsthm, a4, latexsym, amssymb, color, url}
\usepackage{graphicx}

\definecolor{comcolor}{rgb}{0.9,0.3,0.3}
\definecolor{starcolor}{rgb}{0.3,0.3,0.9}
\definecolor{hscolor}{rgb}{0.9,0.6,0.5}
\definecolor{darkgreen}{rgb}{0.1,0.6,0.3}

\newtheorem{thm}{Theorem}[section]
\newtheorem{lemma}[thm]{Lemma}
\newtheorem{corollary}[thm]{Corollary}
\newtheorem{prop}[thm]{Proposition}

\theoremstyle{definition}

\newtheorem{rem}[thm]{Remark}

\newcommand{\be}[1]{\begin{equation}\label{#1}}
\newcommand{\ee}{\end{equation}}
\newcommand{\ba}{\begin{array}}
\newcommand{\ea}{\end{array}}
\newcommand{\bal}{\begin{aligned}}
\newcommand{\eal}{\end{aligned}}

\newcommand{\N}{\mathbb{N}}

\newcommand{\E}{\mathbb{E}}
\newcommand{\p}{\mathbb{P}}
\renewcommand{\P}{\mathbb{P}}

\newcommand{\cB}{\mathcal{B}}
\newcommand{\calC}{\mathcal{C}}

\newcommand{\cD}{\mathcal{D}}
\newcommand{\calE}{\mathcal{E}}

\newcommand{\cP}{\mathcal{P}}

\newcommand{\calT}{\mathcal{T}}

\newcommand{\Var}{\mathrm{Var}}
\newcommand{\cov}{\mathrm{Cov}}

\newcommand{\1}{1\hspace{-0.098cm}\mathrm{l}}

\newcommand{\eps}{\varepsilon}

\newcommand{\ra}{\rightarrow}

\newcommand{\DCM}{{\rm DCM}}
\newcommand{\WCM}{{\rm WCM}}

\begin{document}

\begin{center}
{\Large \bf 
The largest strongly connected component in\\[1mm] Wakeley et\ al's cyclical 
 pedigree model
}\\[5mm]

\vspace{0.7cm}
\textsc{Jochen Blath\footnote{Institut f\"ur Mathematik, Technische Universit\"at Berlin, Stra\ss e des 17. Juni 136, 10623 Berlin, Germany.}, Stephan Kadow\footnotemark[1] and Marcel Ortgiese\footnote{Institut f\"ur Mathematische Statistik,
Westf\"alische Wilhelms-Universit\"at M\"unster, Einsteinstra\ss{}e 62, 48149 M\"unster, Germany.
}
} 
\\[0.8cm]
{\small 28 May 2014} 
\end{center}

\vspace{0.3cm}

\begin{abstract}
  \noindent We establish a link between Wakeley et al's (2012) cyclical pedigree model from population genetics and a randomized directed configuration model (DCM) considered by Cooper and Frieze (2004).  
We then exploit this link in combination with asymptotic results for the in-degree distribution of the corresponding DCM to compute the asymptotic size of the largest strongly connected component $S^N$ (where $N$ is the population size)
of the DCM resp.\ the pedigree.
The size of the giant component can be characterized explicitly (amounting to approximately $80 \%$ of the total populations size) and thus 
contributes 
to a reduced `pedigree effective population size'. 
In addition,
 the second largest strongly connected component is only of size $O(\log N)$.
Moreover, we describe the size and structure of the `domain of attraction' of $S^N$.
In particular, we show that with high probability for any individual the shortest ancestral line
reaches $S^N$ after $O(\log \log N)$ generations, while almost all other ancestral lines take at most $O(\log N)$ generations.

  \par\medskip
\footnotesize
  \noindent \emph{2010 Mathematics Subject Classification}:
  Primary\, 60K35, 
  \ Secondary\, 92D10.%
  \par\medskip
\end{abstract}

\noindent{\slshape\bfseries Keywords:} Directed configuration model, giant component, random digraph, Wakeley et al's cyclical model, coalescence in a fixed pedigree, diploid monoecious Wright-Fisher model.

\section{Introduction}

In a recent article, Wakeley, King, Low and Ramachandran \cite{WKLR12} point out a conceptual flaw in standard coalescent theory when applied to diploid bi-parental organisms. They argue that even in the diploid single-locus case, one should treat the random pedigree as fixed (``quenched'') before letting genetic lineages randomly percolate through them. If the pedigree is unknown, one should average only afterwards. Indeed, at least for the first $\log_2 N$ steps (if $N$ is the population size), significant deviations from the `usual' coalescence probability of lineages, that is, $1/2N$, can be observed.  
However, Wakeley et al also point out, based on a simulation study, that the coalescent model is reasonably good \emph{after} $\log_2 N$ generations. 
More formally, they observe that
$$
\mathbb E \big[{ P}_{\Psi^N}\{\mbox{ two coalesce in generation $k$ } \big| \mbox{ not coalesced in gen $k-1$ }\}\big] \approx \frac{1}{2N}, 
$$
if $k \gg \log_2 N$,
where the inner probability $P_{\Psi^N}$ is taken over the genetic lineages within a \emph{fixed pedigree} $\Psi^N$, while the outer 
expectation averages over such randomly sampled pedigrees.
Moreover, they observe that for $k \ll \log_2 N$, the  expectation deviates drastically from $\frac{1}{2N}$.
Their observation thus hints at a `cut-off phenomenon' for the coalescence probability of random walks on fixed pedigrees after $\log_2 N$ steps. 

In order to isolate the different sources of randomness in their modeling framework, and in order to facilitate a mathematical analysis, they suggest a simplification, namely their so-called `cyclical model' (WCM), in which the exact parental relations in generation 1 are 
re-used to determine the parental relationships in all other previous generations. 

However, note that depending on the realisation of the random ancestral relationships, 
the cyclical model can yield a \emph{disconnected} pedigree, such that some sampled lineages might not coalesce at all. Indeed, although in the majority of simulations Wakeley et al do observe that the coalescence probability for two lineages is close to the Kingman coalescent after $\log_2 N$ steps, in several cases the two lineages fail to coalesce even after a very large number of steps due to sampling them from different connected components (cf.\ their comment on p.~1435 last line in \cite{WKLR12}). 

This fact is intuitively obvious and in line with recent general results of Cooper and Frieze~\cite{CF12}. More precisely, below we will see that the cyclical model can be interpreted as a \emph{directed random graph}, and it is shown in~\cite{CF12}
that a directed random graph with $N$ vertices and i.i.d.\ edge probability $p$ can only be ``completely connected'' (with high probability as $N \to \infty$) if 
$Np \to \infty$. This is not the case for the WCM and raises the question whether the cyclical model, although being mathematically attractive, is a reasonable `toy-model' at all. 

In this article, we thus investigate whether the random directed graph describing the ancestral relationships in Wakeley et al's model $a)$ at least has a `giant component' (with high probability) and $b)$ whether the vast majority of all other vertices have a path leading into this component after few (say at most $\log_2 N$) steps. To this end, we employ results from the theory of directed graphs due to Cooper and Frieze \cite{CF04}, and combine them with an asymptotic analysis of the in-degree distribution of the graph corresponding to WCM, to obtain affirmative and rather detailed answers to $a)$ and $b)$.
This also yields the existence of a unique stationary distribution of the simple random walk on the WCM (concentrated on the giant component).

Our results can be seen as a starting point for a mathematical analysis of the cut-off phenomenon in the cyclical case. Note however that the relevant mathematical literature on random walks on random \emph{directed} graphs (which are the relevant objects here) is much more sparse than for the case of undirected graphs; for a review of the latter see 
e.g.~\cite{RemcoNotes}. One of the few cases considered so far is treated in \cite{CF12}, though only for the rather different `strongly connected' regime ($Np \to \infty$) mentioned above. 
Investigating the presence of the cut-off phenomenon in Wakeley et al's cyclical model is thus an interesting problem
and will be treated in future research.

\section{Cyclical pedigrees as finite digraphs and the directed configuration model} 

\subsection{Definitions and construction}

We start by defining a `pedigree' formally. Suppose we are considering a diploid population of fixed size consisting of $N$ individuals. We label the individuals in generation $r\in \N_0$ as $(r,1), \ldots, (r,N)$, where we count generations going backwards in time starting from the present-time generation $0$.  
The \emph{pedigree} of the population is now modeled by assigning to each individual in generation $r$ two `parents' in generation $r+1$.
For simplicity, we consider a \emph{monoecious} population in which each parent can play the role of both the female and of the male (even simultaneously), so that it is possible that one individual represents both parents
in the previous generation. Abbreviating $[N] := \{1 , \ldots, N\}$, a (mathematical) pedigree $\Psi^N$ is then a directed (multi-)graph on the vertex set
\[ 
V_{\Psi^N} = \{ (r,j) : r \in \N_0 , j  \in [N]  \}   , 
\]
with the property that each vertex $(r,j)$ in generation $r$ has two outgoing edges pointing towards vertices in generation $r+1$.

Wakeley et al.~\cite{WKLR12} suggest the following model that incorporates
the fact that possible genetic lineages 
are restricted to the same pedigree. Assume that a pedigree $\Psi^N$ is fixed. 
We now sample two individuals at random from the population
in generation~0 and trace back their (genealogical) ancestry. 
Mathematically, this corresponds to starting two random walks in (independently sampled) individuals, resp.\ vertices, from $V_{\Psi^N}$ in generation $0$. Then, each random walk can follow one of the precisely two out-going directed edges 
and thus chooses one of the (up to) two parents as specified by $\Psi^N$ with equal probability.
If the two random walks meet, then with probability~$\frac 12$ they coalesce, or otherwise
take different out-going edges to repeat this process until they meet again (this models the fact that we are considering a \emph{diploid} population so that if two genetic lineages meet in one individual they may still remain distinct).
The final aim is to understand the asymptotics of the coalescence probability per generation and the coalescence time in generations for `typical' underlying pedigrees.

Obviously there are many conceivable ways of obtaining/generating a pedigree $\Psi^N$. One
option is to take a deterministic pedigree obtained from data. 
For example, in~\cite{WKLR12} the authors 
consider a pedigree obtained from Swedish family history. However, since the underlying pedigree is often unknown, another possibility
is to sample the pedigree randomly. 
The most obvious model is the \emph{completely randomized model}. Here, in each generation 
every vertex $(r,j)$
samples (independently) two vertices $(r+1,k)$ and $(r+1, k')$ uniformly from the vertices
in generation~$(r+1)$. 
This variant introduces a lot of independent randomness and is thus rapidly mixing. It is at the opposite end of the spectrum in comparison to a deterministic pedigree. Note that some ``mixing properties'' of this model haven been considered in \cite{Chang}.

On the level of simulations it has been shown in~\cite{WKLR12} for various pedigree models
that the resulting (averaged) coalescence probabilities are close to those predicted by Kingman's coalescent after a relatively few number of generations.
In this note, we will be mostly interested 
in the \emph{cyclical pedigree model}, where the ancestral relationships are randomly sampled 
only for the present generation and then re-used for all previous generations. More precisely, for each ``present generation'' individual/vertex~$(0,j)$, we independently sample two individuals $(1,k), (1,k')$ in previous generation~$1$ uniformly 
at random and then connect $(0,j)$ to $(1,k)$ and to $(1,k')$. For each of the consecutive 
generations the ancestral relationships are simply copied from the starting generation. 
Formally, we connect $(r,j)$ to $(r+1,k)$ if and only if $(0,j)$ is connected to $(1,k)$.

The reason that we focus on this model is that we can now
reformulate the problem in terms of \emph{finite} random directed graphs with $N$ vertices. Indeed, the (a priori infinite) pedigree in the cyclical model can be identified with a (finite) directed graph $\WCM_N$ based on the vertices $\{1,\ldots, N\}$
by collapsing all vertices $\{ (r,j ) : r \in \N_0\}$ in the cyclical pedigree $\Psi^N$ into a single vertex $j$ in 
$\WCM_N$ and keeping only the edges from the first generation, see Figure~\ref{fig_cyclical}. Obviously, it is equivalent to run the coalescing random walk (describing the genealogies) on 
either $\WCM_N$ (and counting the number of steps resp.\ generations) or on the cyclical pedigree.\\

\begin{figure}[htp]
	\begin{center}
	\includegraphics[width=12cm]{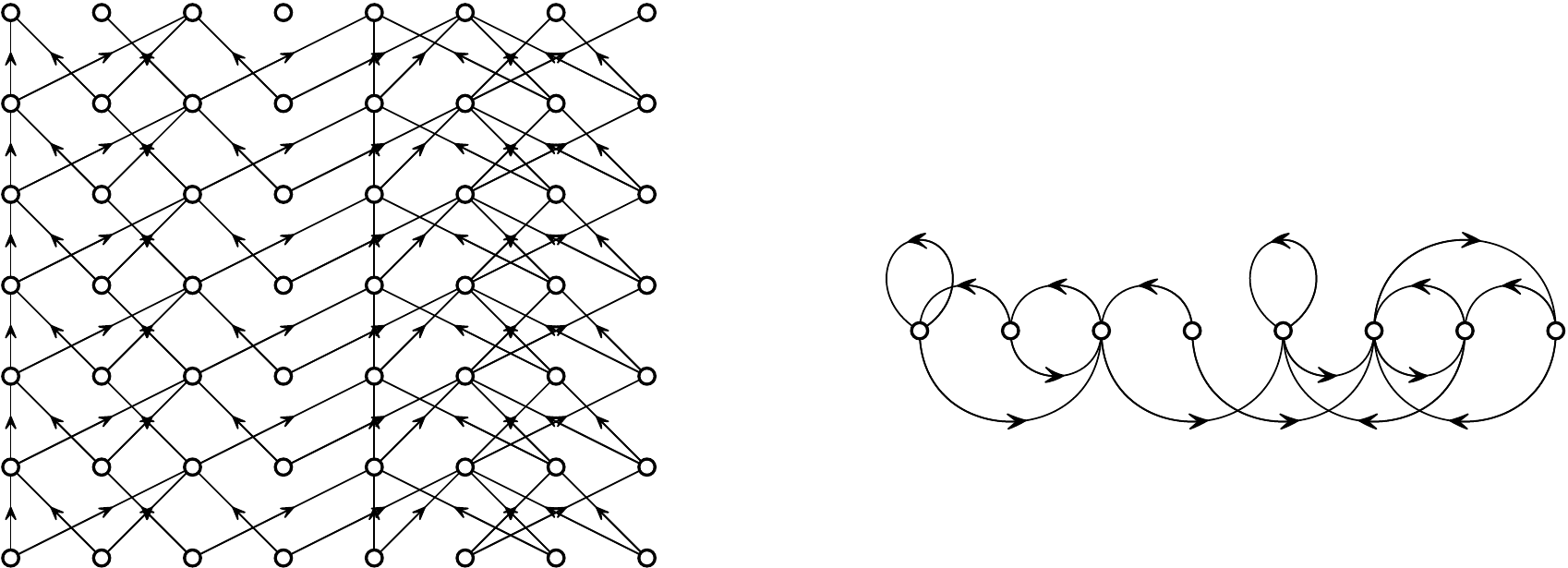}
	\end{center}
	\caption{A realization of the cyclical pedigree (left) and its associated directed graph (right).}
\label{fig_cyclical}
\end{figure}

Since it is possible that two vertices (which may be identical) are connected by two directed edges, we are formally dealing with a directed \emph{multigraph} on the vertex set $\{1, \ldots, N\}$, which we represent by an $N \times N$ array
$\WCM_N = (x_{ij})_{i,j \in [N]}$, where the $x_{ij}$ denote the number of directed
edges pointing from $i$ to $j$, respectively.

We now specify a mechanism to randomly generate cyclical pedigrees resp.\ the corresponding finite digraphs. Throughout the paper, we assume to work on a probability space $(\Omega, \mathcal F, \P)$. Let $(U_i^1, U_i^2)_{i \in [N]}$ be $2N$ independent uniformly distributed random variables on $\Omega$, each taking values in $[N]$. 
Then, if we interpret $U_i^1, U_i^2$ as the
`parents' of vertex $i$, we can realize $\WCM_N$ as 
\begin{equation}
\label{eq:WCM} 
\WCM_N = ( \1_{\{ U^1_i = j \}} + \1_{\{ U^2_i = j \}} )_{i,j \in [N]}. 
\end{equation}
Note that loops and double-edges, while allowed, will be relatively sparse in large populations.

\subsection{Cyclical pedigrees as random directed configuration models}

Our first aim is to show that we can interpret our random digraph $\WCM_N$ given by 
\eqref{eq:WCM} as a randomized version of a so-called \emph{directed configuration model} (DCM), which was popularized by Bollob\'as~\cite{B80} (in the undirected case)
and further investigated by Cooper and Frieze~\cite{CF04} in the directed case.
We first recall the definition of the model for a deterministic degree sequence. Fix $N \in \mathbb N$ and consider the finite vertex set $[N]$. Suppose that we are given so-called in- and out-degree sequences $(d^-_i, d^+_i)_{i \in [N]}$, where  
$d^-_i$ is the (finite) number of in-coming (directed) edges and $d^+_i$ the (finite) number of out-going edges for each vertex $i \in [N]$.  We fix the total number of edges $s\in \N$ and thus assume 
\begin{equation}
\label{eq:1601-1} \sum_{i \in [N]} d^-_i = \sum_{i \in [N]} d^+_i = s. 
\end{equation}
The following procedure generates a random graph, a so-called \emph{directed configuration}, with a pre-specified in- and out-degree sequence: Given $i \in [N]$ and $d^+_i$, consider the set of ``out-half-edges'' of vertex $i$ given by
$$
W^+_i := \big\{w_{i,1}^+, \dots, w_{i, d_i^+}^+\big\}, \quad i \in [N],
$$
and the set of all out-half-edges
$$
W^+ := \bigcup_{i=1}^{N} W_i^+.
$$ 
Similarly, define $W^-_i$ to be the set of in-half-edges of vertex $i$ (of cardinality $d_i^-$), and $W^-$ to be the union of all in-half-edges (of cardinality $s$). We proceed by uniformly matching out- and in-half-edges iteratively: Assume that the elements of $W^+$ are enumerated in some arbitrary order. Connect the first out-half-edge in $W^+$ to a uniformly chosen in-half-edge from $W^-$. Then, connect the second out-half-edge in $W^+$ to one of the (uniformly chosen) remaining in-half-edges, and so on. After $s$ steps, we arrive at a random directed graph $\DCM_N = \DCM_N(d^-_i, d^+_i)$, where
there is a directed edge pointing from vertex $i$ towards vertex $j$ in $\DCM_N$ if and only if one of the out-half-edges from $W^+_i$ 
has been paired up with an in-half-edge from $W^-_j$.

This iterative construction gives rise to nice independence properties within the directed configuration model that we will exploit later. It now seems intuitively obvious (nonetheless we will give a formal proof in Section~\ref{ssn:doublerand} taking care of combinatorial subtleties) that the cyclical model obtained from \eqref{eq:WCM} can be interpreted as a directed configuration model with $s=2N$ and a suitable random in-degree distribution:

\begin{prop}\label{prop:doublerand}
Fix $N \in \N$. Then, the random digraph $\WCM_N$ 
is equal in law to the random digraph $\DCM_N(d^-_i, d^+_i)$ where
$d^+_i := 2$ for all $i\in [N]$
and $d^-$ is distributed as a multinomial vector given by
\begin{equation}
\label{DCM_indegree}
d^- :=(Y_1, \dots, Y_N) \sim \mbox{\rm Mult} \Big(2N, \frac 1N\Big).
\end{equation}
\end{prop}

Note that in $\WCM_N$ the in-degree of a vertex $i \in [N]$ is given by
$$
Y_i := \sum_{k=1}^N \big( {\1}_{\{U_k^1=i\}} + {\1}_{\{U_k^2=i\}} \big) , 
$$
so that it is clear that $(Y_i)_{i \in [N]}$ necessarily has the distribution \eqref{DCM_indegree}.

The component structure of directed configuration models is rather well-understood.
Cooper and Frieze \cite[Section 1.3]{CF04} derive criteria which establish information about the existence and asymptotic size of a giant strongly connected component for the directed configuration model with a deterministic degree sequence.
Proposition~\ref{prop:doublerand} allows us to transfer these results  to the graph $\WCM_N$.

\subsection{Asymptotics for the empirical in-degree sequence of $\DCM_N$}

Consider a directed configuration model $\DCM_N(d^-_i, d^+_i)$ and assume that the
out-degree sequence is constant and equal to $d_i^+ \equiv 2, i \in [N]$. Then, the asymptotic size of the components of $\DCM_N$ (resp.\ the corresponding $\WCM_N$) will only depend on the \emph{empirical in-degree sequence}
$$
\xi^N:=(\xi_k^N)_{0 \le k \le 2N}, \quad \mbox{ where } \quad \xi_k^N:= \frac 1N \sum_{i=1}^N \1_{\{d_i^-=k\}}.
$$
The sequence $(\xi_k^N)_{0 \leq k \leq 2N}$ can be interpreted as the distribution of the in-degree of a uniformly chosen vertex in the graph.

In order to apply the results in~\cite{CF04}, we need  to determine the asymptotic law of the in-degree sequence as $N \to \infty$. Note that the marginal distribution of a multinomially distributed vector in~\eqref{DCM_indegree} is the binomial
distribution with success parameter $\frac{1}{N}$ and number of trials $2N$. In particular, in the limit $N \ra \infty$, the
in-degree distribution $\xi_N^k$ should be close to a Poisson$(2)$ distribution. 
We will state this elementary result in a slightly stronger form that we will need later on, namely as convergence in a suitable uniform $\ell^2$-weighted sense.

\begin{lemma}
\label{thm_1}
Define for convenience $\xi_k^N:=0$ for $k >2N$ and let $(\cP_k^2)_{k \in \N_0}$ be a Poisson$(2)$ distribution. Then, for any $\eps > 0$, as $N \to \infty$,
$$
\P\Big\{ \sum_{k=0}^\infty (k+1)^2 \big|\xi_k^N - {\cal P}^2_k\big| > \varepsilon \Big\} \to 0.
$$
\end{lemma}

We will postpone the proof until Section~\ref{ssn:emp_in_deg}.
The only slight difficulty lies in controlling the correlations between the different observed $\xi_k^N, 0 \le k \le 2N$.

\newpage
\section{The component structure of $\WCM_N$}

In order to formulate our main results, we require some additional notation. 

Let $\{ \calE_N \}_{N \in \N}$ be a sequence of events on $(\Omega, \mathcal F)$. We say that
$\{ \calE_N \}_{N \in \N}$ occurs \emph{with high probability} (whp) if $\lim_{N \ra \infty} \p (\calE_N) =1$. 

Let $G^N$ be a finite directed graph with vertex set $[N]$. We define a \emph{strongly connected component} of $G^N$ to be a maximal subset $S^N$ of $[N]$
such that for each pair of vertices $u, v \in S^N$ there is a directed path (respecting the orientation of edges) from $u$ to $v$ \emph{and} from $v$ to $u$. 
 We say that a family of random graphs $(G^N)_{N \in \N}$ exhibits a \emph{giant strongly connected component} 
of asymptotic relative size $c \in (0,1]$, 
if 
$$
\frac{|S^N|}{N}\stackrel{\p}{\longrightarrow} c \quad \mbox{as } N \ra \infty . 
$$
where $|S^N|$ is number of vertices in $S^N$. 

Consider a Galton-Watson branching process with Poisson$(2)$ offspring distribution $(\cP_k^2)_{k \in \N_0}$.
Since this branching process is supercritical and since $\cP_0^2 > 0$, it is classical that it survives with probability $x^* \in (0,1)$, where $1-x^*$ is the unique fixed point of the probability generating function 
$f$ in $(0,1)$ given by
\begin{equation}
\label{eq:PoisPGF} f(x) = \sum_{k \in \N_0} \cP^2_k x^k = e^{2x - 2} . 
\end{equation}
Note that $x^* \approx 0.797$.
This fixed point will be important in our results below. 
Note that the same branching process plays a closely related role in the work of~\cite{Chang}, in which the time to a common recent ancestor in a bi-parental Wright-Fisher model is investigated.

Adapting the results of~\cite{CF04} for deterministic directed configurations to the random graph obtained from the cyclical model allows us to formulate the following statement about the structure of  $\WCM_N$.

\begin{thm}\label{thm:main}
Let $\WCM_N$ be the random directed graph 
obtained from Wakeley's cyclical model.
\begin{itemize}
\item[(i)] $\WCM_N$ exhibits a unique giant strongly connected component $S^N$
with asymptotic size 
$$
\frac{|S^N|}{N} \stackrel{\p}{\longrightarrow} x^*\quad \mbox{ as } N \to \infty,
$$
where $x^*$ is the survival probability of a Galton-Watson tree with Poisson$(2)$ offspring distribution. 
\item[(ii)] With high probability the following holds: For every vertex $v \notin S^N$, there is a directed path from $v$ to $S^N$ consisting of at most
$$
\frac{\log \log N}{\log 2} (1 + o(1)) 
$$
edges. However,  there are no directed 
edges pointing from $S^N$ to its complement.
Moreover,  the number of vertices that can be reached from $v$ without going through $S^N$ is bounded above
by 
\[ 
\frac{2}{-\log (4x^*(1-x^*))} \log N (1+o(1)) . 
\]
In particular, the same bound applies to the size of the second largest strongly connected component.
\end{itemize}
\end{thm}

\begin{rem}[Shape of the complement of $S^N$] 
Theorem~\ref{thm:main} gives a rather detailed description of the directed graph $\WCM_N$, 
see also Figure~\ref{fig_giant}.
Outside the giant component there are many small components of at most logarithmic size, whose vertices 
have connections leading quickly towards the giant component.
 Moreover, as expected for sparse random graphs, 
 we will see in the proofs in Section~\ref{ssn:adaptation}
that the local neighbourhood of a typical vertex has an almost tree-like structure.
In particular, we show that the number of cycles obtained after tracking the out-going edges for the first few generations starting in a particular vertex is negligible. 
Any long cycles that ensure overall connectivity are contained within the giant component.
\hfill $\diamond$
\end{rem}

\begin{figure}[htp]
	\begin{center}
	\includegraphics[width=10cm, clip=true, trim=0 2cm 0 1cm]{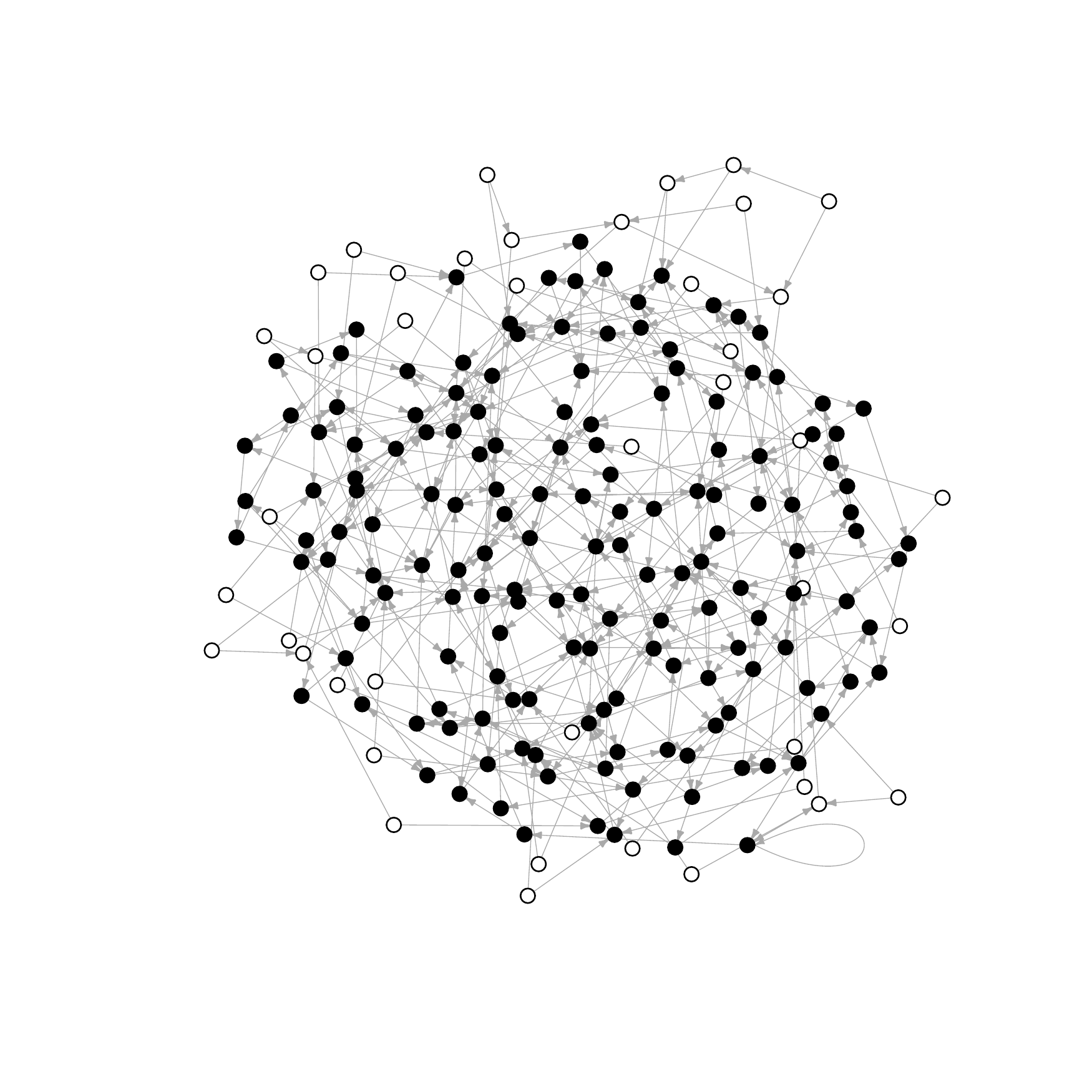}
	\end{center}
	\caption{A simulation of the directed graph $\WCM_N$ for $N = 200$ using
	the igraph package for R, see {\tt igraph.org}. The vertices in the 
	giant strongly connected component are black, while the vertices in the complement are white.}
\label{fig_giant}
\end{figure}

Theorem~\ref{thm:main} has important implications for 
the analysis of the ancestral lineages modeled by random walks
within the pedigree.

\begin{corollary}
With high probability, the simple, symmetric random walk on the directed graph $\WCM_N$
has a unique stationary distribution. Moreover, the stationary distribution is supported on the giant strongly connected component.
\end{corollary} 

\begin{proof} This follows immediately from the fact that, with high probability, the strongly connected component $S^N$ can be reached  from every vertex and there are no connections going out from $S^N$. Hence, $S^N$ is a unique essential communicating
class for the random walk, which guarantees uniqueness and existence of the stationary distribution, see e.g.~\cite[Prop.~1.26]{LPW09}.
\end{proof}

\begin{rem}[``Pedigree effective population size'']
Theorem~\ref{thm:main} also gives us a first indication on the mixing time.
In particular, part (ii) tells us that it is very likely that a random walk started in an arbitrary point on $\WCM_N$ ends up in the strongly connected
component after no more than $\frac{2}{-\log (4x^*(1-x^*))} \log N (1+o(1))$ steps. 
Reaching $S^N$ is obviously a  necessary condition to 
meet with another random walk started from a `different' part of the complement of $S^N$.
Since the size of the giant strongly connected component, in which coalescences happen, is 
approximately $x^*N$, i.e.\ to about $80 \%$
of the population, this indicates that we could observe a reduction to a `pedigree effective population' size $N_{e} < N$. However, it is not obvious how to quantify $N_e$. Working with $N_e=x^* N$ may be inadequate, since vertices in $S^N$ are biased towards higher in-degrees. \hfill $\diamond$
\end{rem}

On a heuristic level, it is not hard to understand the occurrence of branching processes in this context. Consider a uniformly chosen vertex $v$ in $\WCM_N$ and explore its `neighbourhood' respecting the direction of edges. This can be done in two \mbox{(time-)} directions:
either following only the out-going edges (in a breadth-first manner), leading to the \emph{fan-out} $R^+(v)$, i.e.\ all vertices that can be reached from $v$. Since the out-degree of each vertex is always $2$, at the beginning of the exploration one essentially sees a binary tree (whp). 
The second option is to follow the in-coming edges in reverse (time-) direction leading to 
$R^-(v)$, the \emph{fan-in} of vertex $v$.
Since the in-degree distribution of a uniformly chosen vertex is by Lemma~\ref{thm_1} 
approximately Poisson$(2)$, the fan-in initially resembles a Galton-Watson branching process
with a Poisson offspring distribution. In particular, the probability that the fan-in is large
is close to the survival probability $x^*$. 
As shown in~\cite{CF04}, the strongly connected component is essentially given by those vertices that
have both a large fan-in and a large fan-out.
In particular, in our set-up (where each fan-out is large whp), the strongly connected component consists of those vertices with a large fan-in, whose proportion is given by~$x^*$.

Part (i) of Theorem~\ref{thm:main} is a straight-forward application of the results in~\cite{CF04} for the deterministic case once we can show that our random degree sequences
 $(d^-,d^+)$ obtained from $\WCM_N$ are \emph{proper} in a suitable sense which we
will state and prove in Section~\ref{ssn:adaptation}. 
Part (ii) is a  strengthening of the techniques in~\cite{CF04} for our special case and
its proof will  be carried out in Section~\ref{ssn:strengthening}.

\section{Proofs}

\subsection{Proof of Proposition~\ref{prop:doublerand}}
\label{ssn:doublerand}

Recall that we can represent a directed (multi-)graph $G=G^N$ on the vertices $[N] := \{1, \ldots, N\}$ as a matrix $(x_{ij})_{i,j \in [N]}$,
where $x_{ij}$ denotes the number of (directed) edges pointing from vertex $i$ to $j$. In the following, we simply refer to $G$ as a graph.
Assume that we are given in- and out-degree sequence $(d^-_i, d^+_i)$ such that 
\begin{equation}
\label{eq:1601-1b} 
\sum_{i \in [N]} d^-_i = \sum_{i \in [N]} d^+_i = s , 
\end{equation}
where $s \in \N$ denotes the total number of edges.
As a first result we compute the probability that a (uniformly sampled) $\DCM_N$ corresponding to a given degree distribution $(d^-_j, d^+_j)_{j \in [N]}$
is equal to a given graph $G$.

\begin{lemma} 
Given a graph $G = (x_{ij})_{i,j \in [N]}$ such that 
\[ 
d_i^- = \sum_{j \in [N]} x_{ji} \quad\mbox{and}\quad d_i^+ = \sum_{j \in [N]} x_{ij} 
\]
 and such that~\eqref{eq:1601-1b} holds, we have that
\begin{equation}
\label{eq:1601-2} 
\p \big\{ \DCM_N ((d^-, d^+)) = G \big\}  = \frac{ \prod_{i \in [N]} d_i^-! \prod_{i \in [N]} d_i^+!}{s! \prod_{i,j \in [N]} x_{ij}!} . 
\end{equation}
\end{lemma}

\begin{proof} Recall that in the construction of $\DCM_N$, we find a configuration by uniformly matching the out-half-edges in 
$W^+$ with the in-half-edges in $W^-$.
Since each of the $W^-$ and $W^+$ contain $s$ half-edges, there are $s!$ possible directed configurations, 
so that
\[
\p \big\{ \DCM_N ((d^-, d^+)) = G \big\}  = \frac{N(G)}{s!} , 
\]
where $N(G)$ denotes the number of configurations that yield the same $G$. Therefore, we need to count $N(G)$. First we note that permuting the in-half resp.\ out-half-edges does not change the resulting graph. The number of these permutations is $\prod_{i \in [N]} d_i^-! \prod_{i \in [N]} d_i^+!$. However, if there are multiple edges between the vertices $i$ and $j$, then permuting these
does not give a new configuration. Therefore, we need to compensate correspondingly and thus we obtain
\[ N(G) = \frac{ \prod_{i \in [N]} d_i^-! \prod_{i \in [N]} d_i^+! }{ \prod_{i,j \in N} x_{ij}!} , \]
which gives the stated result.
\end{proof}

In our setting, for the graph $\DCM_N$, we can specialise to the situation that the total number of edges is $s = 2N$ and 
also $d^+_i = \sum_{j \in [N]} x_{ij} = 2$ for all $i$. Moreover, $x_{ij}$ can only take the values $\{0, 1, 2\}$ and thus
\begin{equation*} \prod_{i, j \in [N]} x_{ij}! = 2^{n(G)}, \end{equation*}
where $n(G)$ is the  number of 
vertices $i$ in $G$ that connect via two edges to another vertex.
Therefore, for our model, the probability in~\eqref{eq:1601-2} simplifies to
\begin{equation}\label{eq:1601-4}
\p \big\{ \DCM_N ((d^-, d^+)) = G \big\}  = \frac{ \prod_{i \in [N]} d_i^-! \,2^{N - n(G)} }{(2N)! }. 
\end{equation}

\begin{proof}[Proof of Proposition~\ref{prop:doublerand}]
As before denote by $(Y_i)_{i \in [N]}$ a vector that is multinomially distributed with parameters
$2N$ (number of trials) and success probabilities $\frac{1}{N}$ for each of the $N$ categories.

Then, we can state Proposition~\ref{prop:doublerand} as saying that for any $G = (x_{ij})_{i,j\in [N]}$,
\begin{equation}\label{eq:1601-3} \p \{ \WCM_N = G \} = \E \Big[ \left.\p \{ \DCM_N((d^-,2)) = G\}\right|_{d^-_i = Y_i} \Big]  , 
\end{equation}
where we only need to consider graphs $G$ with $\sum_{j} x_{ij} = 2$. 

First, we calculate the left-hand side. Recall from the construction of $\WCM_N$ in~\eqref{eq:WCM}  that $U_i^1, U_i^2$ are independent random variables that
are uniformly distributed on $[N]$ and denote the vertices  that vertex $i$ connects to
in $\WCM_N$. Thus, we know that
\[ 
\begin{aligned} \p \{ \WCM_N = G \} & = \prod_{i \in [N]} \p \big\{ \1_{\{U_i^1 = j\}} + \1_{\{ U_i^2 = j \}} = x_{ij} \forall j \in [N] \big\} \\
& =     \prod_{i \in [N]} \frac{2}{ \prod_{j \in [N]} x_{ij}! } \Big(\frac{1}{N}\Big)^2 
= \frac{2^N N^{-2N}}{ \prod_{i,j\in [N]} x_{ij}!} = 2^{N - n(G)} N^{-2N}, 
   \end{aligned}
\]
where we used that the random variable $\1_{\{U_i^1 = j\}} + \1_{\{ U_i^2 = j \}}$ is multinomially distributed with parameters $2$ for the number of trials for $N$ categories with success probability~$\frac{1}{N}$ each.

Now, we consider the right-hand side of~\eqref{eq:1601-3}. We note that we can express 
the in-degree as $d^-_i = \sum_{j \in [N]} x_{ji}$ in terms of the given graph $G$, so we obtain
using~\eqref{eq:1601-4}
\[ \begin{aligned} \E \Big[ \Big.\p \{ &\DCM_N((d^-,2))    = G \}\Big|_{d^-_i = Y_i} \Big] \\
& = 
\p \big\{ \DCM_N((d^-,2)) = G \big\} \p \big\{ Y_i = d^-_i \ \forall i \in [N] \big\} \\
& = 
\frac{\prod_{i \in [N]} d_i^-! \, 2^{N-n(G)}}{(2N)!} \ \frac{(2N)!}{\prod_{i \in [N]} d_i^-! } \Big(\frac{1}{N}\Big)^{2N}  
= 2^{N - n(G)} N^{-2N} . 
 \end{aligned} \]
Hence, we see that both sides of~\eqref{eq:1601-3} agree.
\end{proof}

\subsection{Convergence of the empirical in-degree sequence}\label{ssn:emp_in_deg}

In this section, we prove a slightly stronger version of  Lemma~\ref{thm_1}. We recall that
the empirical in-degree sequence is given as
\[ 
\xi^N_k = \frac{1}{N} \sum_{i= 1}^N \1_{\{ Y_i = k \}} ,
\]
where $(Y_i)_{i \in [N]}$ is multinomially distributed with parameters $2N$ and $\frac{1}{N}$.
We will show that the empirical degree distribution $(\xi^N_k)_{k \in \N_0}$ converges
to a Poisson$(2)$ distribution $(\cP_k^2)_{k \in \N_0}$ in the sense that
\[ 
\sum_{k \in \N_0} \ell_k| \xi^N_k  - \cP_k^2| \stackrel{\p}{\longrightarrow} 0 , \quad \mbox{as } N\ra \infty ,
\]
where $\ell_k$ is an increasing sequence such that $0 < \ell_k \leq e^{Ck}$ for some $C > 0$ and where we set $\xi^N_k = 0$ for $k > 2N$. This statement then easily implies Lemma~\ref{thm_1}.

\begin{proof}[Proof of Lemma~\ref{thm_1}]
We write $(\cB^N_k)_{k \in [2N]}$ for a binomial distribution with parameters $\frac 1N$ (success probability)
and $2N$ (number of trials).
First recall that by the Poisson limit theorem, for each $k$,
\begin{equation}\label{eq:Poisson}
\E\big[\xi_k^N\big] = \cB_k^N \to {\cal P}^2_k, \quad \mbox{as } N \ra \infty.
\end{equation}
To simplify notation abbreviate
$$
{\cal B}_{\ge k}^N:= \sum_{i=k}^{2N} {\cal B}_i^N \quad \mbox{ and } \quad \xi_{\ge k}^N := \frac 1N \sum_{i=1}^N \1_{\{Y_i \ge k\}}, \quad 0\le k \le 2N.
$$
We will need the following  straight-forward consequence of the Chernoff bound for i.i.d.\ Bernoulli variables, 
\begin{equation}\label{eq:chernoff} \cB_{\geq k}^N \leq  \frac{(2e)^k}{k^k} \quad \mbox{for all } k \in \{ 2, \ldots, 2N\} , \end{equation}
where the latter bound is the same as for the tail of a Poisson$(2)$ random variable.
Then, observe that
\begin{equation}\label{eq:2602-1} \sum_{k \in \N_0} \ell_k | \xi_k^N - \cP_k^2| \leq 
\sum_{k \in \N_0} \ell_k | \xi_k^N - \cB_k^N| + \sum_{k \in \N_0} \ell_k| \cB_k^N - \cP_k^2| . \end{equation}
Moreover, the second (deterministic) sum on the right hand side converges to $0$ by the Poisson approximation~\eqref{eq:Poisson} combined with the tail bound~\eqref{eq:chernoff} (and an analogous bound for the Poisson distribution). Therefore, it remains to show that the first sum on the right hand side of~\eqref{eq:2602-1} converges 
to $0$ in probability.
To this end, consider
\begin{equation}
\label{eq:2602-2}
\begin{aligned} \sum_{k \in \N_0} \ell_k |\xi_k^N - \cB_k^N| & \leq \sum_{k \in \N_0} \ell_k |\xi_{\geq k}^N - \cB_{\geq k}^N|
+ \sum_{k \in \N_0} \ell_k |\xi_{\geq k+1}^N - \cB_{\geq k+1}^N| \\
& \leq 2 \sum_{k \in \N_0}   \ell_k |\xi_{\geq k}^N - \cB_{\geq k}^N|, 
\end{aligned}
\end{equation}
where we used that $\ell_k$ is increasing.

We need to control the correlation of the variables $(Y_i)_{i \in [N]}$ used in the definition of  $\xi_k^N$.  However, it is well-known (see e.g.~\cite{JDP83}) that the components of a multinomially distributed vector are `negative quadrant dependent'  in the sense that for $i \neq j$
\[ 
\cov( \1_{\{Y_i \geq k \}} , \1_{\{ Y_j \geq k\}} ) \leq 0 , 
\]
for any $k  \in [2N]$. In particular, it follows that
\[\begin{aligned} \Var(\xi_{\geq k}^N) = \Var\Big( \frac{1}{N} \sum_{i =1}^N \1_{\{ Y_i \geq k \}}\Big) 
\leq \frac{1}{N^2} \sum_{i =1}^N \Var(\1_{\{Y_i \geq k\}} ) 
 \le \frac{{\cal B}^N_{\ge k}}{N}.
\end{aligned}\]
Hence, we find that by Chebyshev's inequality and the bound~\eqref{eq:chernoff} in the last step
\[ 
\begin{aligned} 
\p \Big\{ \sum_{k \in \N_0} \ell_k | \xi_{\ge k}^N - \cB_{\ge k}^N|  > \eps \Big\}  
& \leq \sum_{k \in \N_0}
\p \Big\{ | \xi_{\ge k}^N - \cB_{\ge k}^N | >  \frac{\eps}{2^{k+1}\ell_k} 
 \Big\} \\
& \leq \sum_{k \in \N_0} 2^{2k+2} 
(\ell_k)^2 \frac{1}{\eps^2} \Var( \xi_{\geq k}^N) \\
& \leq \frac{1}{\eps^2 N} \sum_{k \in \N_0} 2^{2k+2} 
e^{2Ck} 
\cB_{\geq k}^N \\
&\leq
 \frac{1}{\eps^2 N} \sum_{k \in \N_0} 2^{2k+2}  
e^{2Ck} 
 (2e)^k e^{-k \log k}, 
 \end{aligned}
 \]
which converges to $0$ as $N \ra \infty$ as the sum on the right hand is finite.
Therefore, we have completed the proof using~\eqref{eq:2602-1}
and~\eqref{eq:2602-2}.
\end{proof}

\subsection{Adapting the results from the directed configuration model}\label{ssn:adaptation}

The results in~\cite{CF04} are stated for a directed configuration 
model with \emph{deterministic} degree sequence $(d^-_i, d^+_i)_{i \in [N]}$ which is assumed to be sufficiently regular (`proper').
In our case, where the out-degree is constant, i.e.\ $d^+_i = 2$ for all $i \in [N]$, the regularity conditions can be simplified somewhat. In order to avoid extra notation, we state the definition only in this special case, although we emphasize that the results in~\cite{CF04} are more general.

As above denote by $(\xi_k^N)_{k = 0}^{2N}$ the in-degree distribution of a uniformly chosen vertex. 
Moreover, let $\Delta_N$ be the maximal out- or in-degree in the graph $\DCM_N$.
We consider a degree sequence $(d^-_i, d^+_i)_{i \in [N]}$ with $d^+_i = 2$ for all $i \in [N]$
(so that in particular the total number of edges is equal to $2N$). This degree sequence is called \emph{proper}, if the 
following conditions hold: 
\begin{itemize}
\setlength{\itemsep}{0pt}
\item[(C1)] There exists a constant $K > 0$ (not depending on $N$) such that
\[ 
\sum_{k \in [2N]} k^2 \xi_k^N \leq K\] 
\item[(C2)] The maximal degree is not too large in the sense that \[ \Delta^N \le \frac{N^{1/12}}{\log N}.
\]
\end{itemize}

\begin{rem} The original  definition in~\cite{CF04} for a degree sequence with a non-trivial out-degree 
also involves conditions on the total number of edges and the average directed degree that 
are trivially satisfied here. Moreover, the mixed moments of the in- and out-degree distribution should not be too large, 
which in our case is covered by~(C1).
\end{rem}

In order to transfer the results to the graph $\WCM_N$, we need to sample the in-degree sequence
$(d^-_i)_{i \in [N]}$
as described in Proposition~\ref{prop:doublerand} and then check that the resulting degree sequence
is proper with high probability.

\begin{lemma}
\label{proper}
Let $(d^-_i, d^+_i)_{i \in [N]}$ be the random degree sequence corresponding to
$\WCM_N$. Then
$$ 
\P \big\{\, (d^-_i, d^+_i)_{i \in [N]} \mbox{ is proper} \,\big\} \to 1 \quad \mbox{ as } \quad N \to \infty.
$$
\end{lemma}

\begin{proof}
Recall that the in-degree sequence is given by $d^-_i =: Y_i$, where $(Y_i)_{i \in [N]}$ is multinomially distributed with parameters $2N$ and $\frac{1}{N}$. 

The first condition (C1) follows from the convergence of the empirical in-degree distribution
stated in Lemma~\ref{thm_1} and since a Poisson$(2)$ random variable
has finite mean and variance~$2$.

For the second condition (C2), we show the much stronger statement that 
\begin{equation}
\label{fourth}
\P \big\{ \Delta^N \le \log N \big\} \to 1 \quad \mbox{ as } N \to \infty.
\end{equation}
Since the out-degree for each vertex in $\WCM_N$ is always 2, we can concentrate on the in-degrees. Note that, by exchangeability and by the Chernoff 
bound as in~\eqref{eq:chernoff}
$$
\P \Big\{ \max_{1 \le i \le N} Y_i > \log N\Big\} \le N \P \big\{Y_1 > \log N\big\} \le \frac{N^{\log 2 +2}}{ N^{\log \log N}} \to 0
$$
as $N \to \infty$, which implies \eqref{fourth}.
\end{proof}

\begin{proof}[Proof of Theorem~\ref{thm:main}(i)]
Lemma~\ref{proper} together with Proposition~\ref{prop:doublerand} allow us to 
apply~\cite[Theorem~1.2]{CF04}, which shows that there is a unique giant strongly connected component~$S^N$, 
which satisfies for any $\eps > 0$,
\[ 
\p \big\{ (1-\eps) x_N^* N \leq |S^N|  \leq (1+\eps) x_N^* N \big\} \ra 1 \quad \mbox{as } N \ra \infty . 
\]
Here, $1-x_N^* $ is the extinction probability of a Galton-Watson process with offspring distribution $(\xi_k^N)_{k \in [N]}$, 
which is given by the unique fixed point in $(0,1)$ of the probability generating function of $(\xi_k^N)$
\[ 
f^N(x) = \sum_{k \in \N_0} \xi_k^N x^k . 
\]
However, by Lemma~\ref{thm_1}, we have that if $f$ is the probability generating function of the Poisson$(2)$ 
distribution as defined in~\eqref{eq:PoisPGF}, then 
\[ \sup_{x \in [0,1]} | f(x) - f^N(x)| \leq \sum_{k \in \N_0} | \xi_N^k - \cP^2_k| \stackrel{\p}{\longrightarrow} 0 . \]
Since the same uniform convergence holds for $(f^N)' \ra f'$, we can easily deduce that 
the fixed point $1-x^*_N$ of $f^N$ converges in probability to $1-x^*$, the unique fixed point of $f$ in $(0,1)$. Therefore, we have shown part (i) of Theorem~\ref{thm:main}.
\end{proof}

\subsection{Extensions for constant out-degree}\label{ssn:strengthening}

In order to prove the  second part of Theorem~\ref{thm:main}, we need to 
recall further details from the proof strategy of~\cite{CF04}.
In the following, we will assume that the degree sequence is fixed, satisfies $d^+_i = 2$ for all $i \in [N]$ and that it is proper in the above sense, i.e.\ (C1) and (C2) hold.

Recall that for a given vertex $v$ its \emph{fan-out} $R^+(v)$ consists of all vertices
that can be reached from $v$ along directed paths. Similarly, its fan-in $R^-(v)$ consists 
of all vertices $u$ so that $v$ can be reached from $u$. One reason why the configuration model
is so amenable to analysis is that the corresponding digraph can be constructed by consecutively exploring the fan-outs of its vertices step-by-step.  Indeed, for any vertex one starts to
build up its forward neighbourhood by first uniformly and independently matching its out-half-edges with the in-half-edges of the next generation of vertices, then the out-half-edges of those vertices (say in a breadth-first manner) and so on. 
Upon completion of the forward neighbourhood $R^+(v)$, the remainder of the graph can be constructed by independently repeating this procedure for the remaining vertices and the remaining half-edges.

A central result, see~\cite[Theorem~2.1]{CF04}, is the following dichotomy, where $1-x^*_N$ is the extinction probability
of Galton-Watson process with offspring distribution $(\xi_N^k)$ (as above), which we combine
with the statement of~\cite[Lemma 4.3]{CF04}.

\begin{prop}[\cite{CF04}]\label{prop:F1} There exists a constant $A_0$, such that with high probability, the following holds:
\begin{itemize} 
	\item[(i)] For all vertices $v$, the number of edges in $R^+(v)$
	is either less than $A_0 (\Delta^N)^2 \log N$, or equal to $2x_N^* N + O(\Delta^N \sqrt{ N\log N})$.
	Moreover, the number of edges in $R^-(v)$
	is either less than $A_0 (\Delta^N)^2 \log N$ or equal to $2  N + O(\Delta^N \sqrt{N \log N})$.
	\item[(ii)] If $|R^+(u)|, |R^-(v)| > A_0 (\Delta^N)^2 \log N$, then $R^+(u) \cap R^-(v) \neq \emptyset$.
\end{itemize}
\end{prop} 
We say that a vertex has a \emph{small fan-in} (resp.\ fan-out) if the first (logarithmic) bound holds and otherwise say it has a \emph{large fan-in} (fan-out). 
As we will recall in the proof of Proposition~\ref{prop:ext}, 
these two properties allow us to show that with high probability, the maximal strongly connected component
consists of those vertices with a large fan-in and a large fan-out.

For the proof of Theorem~\ref{thm:main}, we also need the following technical result.
	
\begin{lemma}[{\cite[Lemma 5.1]{CF04}}]\label{le:F2}
With high probability, simultaneously for all vertices $v$:
\begin{itemize}
\item[(i)]  In the first $(\Delta^N \log N)^2$ steps of the (breadth-first) forward exploration process of $R^+(v)$ described above, in all but possibly one step a new vertex is added.
\item[(ii)] If $R^+(v)$ is large, then the number of edges pointing towards $R^+(v)$ from its complement
 is 
$2x^*(1-x^*)N(1+o(1))$.
\end{itemize}
\end{lemma}

\begin{proof}
(i) For the reader's convenience, we recall the proof from~\cite{CF04}.
Fix a vertex~$v$ and denote by $A(s)$ the vertices discovered and by $U^-(s)$ the in-half-edges used
after $s$ steps of the forwards exploration
starting from vertex $v$. Then, define 
$$
J(s) = \bigcup_{u \in A(s)} W_u^- \setminus U^-(s)
$$
as the set of  unused in-half-edges incident to vertices in $A(s)$. 
The event that we form an edge in step $s+1$ pointing  towards a previously discovered
vertex corresponds to connecting the next out-half-edge 
with an in-half-edge in $J(s)$. The latter has probability
\[ 
\frac{|J(s)|}{2N - s} \leq \frac{\Delta^N s}{2N -s} \leq \frac{ (\Delta^N)^3 \log^2 N }{2N - (\Delta^N \log N)^2}, 
\]
provided $s \leq (\Delta^N \log N)^2$. 
In particular, it follows that the probability that during the first $(\Delta^N \log N)^2$ steps of the exploration
this event happens at least twice is bounded from above by
\[ \Big( \frac{ (\Delta^N)^3 \log^2 N }{2N - (\Delta^N \log N)^2} \Big)^2 . \]
This expression still converges to $0$ after summing over all possible starting vertices~$v$.

(ii) In the proof of~\cite[Lemma 4.3]{CF04} it is shown that the number of edges pointing towards
$R^+(v)$ from its complement is of order $C N(1+o(1))$, where
\[\bal C &  := 2 \Big( 1 -x^* - \frac 12  \sum_{k \geq 0} k \cP_k^2 (1-x^*)^k \Big)\\
& =
 2 \Big( 1 -x^* - \frac 12  \sum_{k \geq 0} k e^{-2} \frac{2^k}{k!} (1-x^*)^k \Big)\\
 & =
 2 ( 1 -x^*)\Big( 1 - \sum_{k \geq 0} \cP_k^2(1-x^*)^k\Big) = 
2 x^* (1-x^*),
\eal  \]
where we used in the last step that $1-x^*$ is by definition a fixed point of the probability generating function of $(\cP_k^2)$.
\end{proof}

In particular, Lemma~\ref{le:F2} implies that with high probability, the forward exploration process of each $v$
produces at most one cycle (even in the undirected sense) after the first $(\Delta^N \log N)^2$ steps.

The statement (ii) of Theorem~\ref{thm:main} now follows from the next two Propositions~\ref{prop:ext}
and~\ref{prop:new}. The first proposition collects some straight-forward consequences of the construction
of the giant strongly connected component in~\cite{CF04}.

\begin{prop}\label{prop:ext}
For any degree sequence that is proper and has constant out-degree $2$, 
with high probability, $\DCM_N$ with largest strongly connected component $S^N$ 
satisfies:
\begin{itemize}
\item[(i)]  $S^N$ consists of all vertices with a large fan-in.
\item[(ii)] For any vertex $v \notin S^N$ there is a directed path from 
$v$ to $S^N$. 
\item[(iii)] There are no directed edges leading away from $S^N$.
\end{itemize}
\end{prop}

\begin{proof} We adapt the proof of Corollary 4.4 in~\cite{CF04} to our special case.
Let $\cD_N$ be the set of digraphs that satisfy the conclusions of Proposition~\ref{prop:F1}
and Lemma~\ref{le:F2}. Then $\p (\DCM_N \in \cD_N) \ra 1$ as $N \ra \infty$, 
so it suffices to prove statements (i)-(iii) for any digraph $\DCM_N \in \cD_N$.

We first show that if $\DCM_N \in \cD_N$, then any vertex $v$ has a large fan-out.
Suppose that for a vertex $v$ in $\DCM_N$ the forward exploration $R^+(v)$ terminates after
$s$ steps and $s \leq A_0 (\Delta^N)^2 \log N$. 
Then, by Lemma~\ref{le:F2} at least $s$ distinct vertices (including the starting vertex $v$)
were discovered, 
and since each vertex has out-degree $2$, for the exploration to be completed 
at least $2s$ out-half-edges have
to have been connected.
This contradicts the construction of the exploration process, where in each step only a single edge is created.
Thus, by  the dichotomy in Proposition~\ref{prop:F1}(i),
any  vertex $v$ in the graph has to have a large fan-out.

Define $\tilde S^N$ as the set of all vertices with a large fan-in.
We claim that $\tilde S^N$ is a maximal strongly connected set.
Indeed, we know that any vertex $u$ (and so in particular any vertex in $\tilde S^N$)
has a large fan-out. So  since any $v \in \tilde S^N$ has a large fan-in, 
 by Proposition~\ref{prop:F1}(ii) 
it follows that $R^+(u) \cap R^-(v) \neq \emptyset$ and therefore
  $u$ is connected to $v$.
  This implies 
 $\tilde S^N$ is strongly connected and also statement (ii) once we have shown 
that the largest strongly connected component $S^N$ is equal to $\tilde S^N$.

Furthermore, if there is an edge connecting $\tilde S^N$ to some vertex $w$, then
$w$ has a large fan-in (it has an incoming edge from a vertex with a large fan-in). Therefore, 
$w \in \tilde S^N$. 
This implies that $R^+(v) \subseteq \tilde S^N$ for any $v \in \tilde S^N$ and also that 
$\tilde S^N$ is maximal strongly connected. 
Furthermore, this shows claim (iii) of the proposition (again  once we have established
that $\tilde S^N = S^N$).

Now, for $v\in \tilde S^N$ we always have $\tilde S^N \subseteq R^+(v)$  so that we can 
 deduce from the previous step that  $\tilde S^N = R^+(v)$.
Note that, since the out-degree is always $2$,  the mapping that associates to each directed edge in 
(the subgraph induced by) $R^+(v)$ the starting vertex of the edge is two-to-one. 
 Thus, we can deduce from the dichotomy of Proposition~\ref{prop:F1} that 
$|\tilde S^N| = |R^+(v)| = x^* N (1+o(1))$ for any $v \in \tilde S^N$.

 Suppose that there exists a further strongly connected component 
$\hat S^N$, say of size $|\hat S^N| \geq \eps N + 1$ for 
some $\eps > 0$. Let $v \in \hat S^N$. 
Again, for any vertex $w \in R^-(v)\setminus \{v\}$ there is at least one edge in (the 
induced subgraph) of $R^-(v)$ with starting vertex~$w$, therefore
the number of edges in $R^-(v)$ is bounded from below by $ |R^-(v)| -1 \geq |\hat S^N|-1 \geq \eps N$. 
By the dichotomy of Proposition~\ref{prop:F1}, the fan-in $R^-(v)$ is large
so that $v \in  \tilde S^N$ 
and thus $\hat S^N = \tilde S^N$.
In particular, we have seen that the largest strongly connected component $S^N$
is necessarily equal to $\tilde S^N$, proving (i) and thus completing the proof.\end{proof}

The final proposition considerably extends the arguments of~\cite{CF04} and gives rather explicit quantitative results about the complement of the  giant component. Here, we make extensive use of the  fact that in our special case the out-degree in $\DCM_N$ is always $2$.

\begin{prop}\label{prop:new}
For any degree sequence that is proper and has constant out-degree $2$, 
with high probability, $\DCM_N$ with largest strongly connected component $S^N$ 
satisfies:
\begin{itemize}
\item[(i)] For any vertex $v \notin S^N$, the length of the shortest path from $v$ leading into $S^N$
is at most $\frac{\log \log N }{\log 2}(1+o(1))$.

\item[(ii)] For any vertex $v \notin S^N$, $|R^+(v) \setminus S^N| \leq \frac{2}{-\log(4x^*(1-x^*))} \log N (1+o(1))$.
\item[(iii)] The number of vertices in the second largest strongly connected component is bounded
by $\frac{2}{-\log(4x^*(1-x^*))} \log N (1+o(1))$.
\end{itemize}
\end{prop}

\begin{proof}
(i) Let $\eps > 0$ and define  $\ell = (1+\eps) \frac{\log \log N}{\log 2}$.
Suppose that there exists a vertex $v$ such that every path of length $\ell$ in $R^+(v)$ starting in $v$
does not intersect with the giant strongly connected component $S^N$. 

Let $s_\ell = 2(2^{\ell} -1)$.  Denote by $A_v(s_\ell)$ the vertices found in the forward exploration (in a breadth-first way) up to step $s_\ell$ starting from $v$.
Every vertex has out-degree $2$ and by Lemma~\ref{le:F2}(i) in all but possibly one of the first $s_\ell$  steps a previously unseen vertex is discovered in the exploration.
Therefore, if we omit the possible edge leading to an already discovered vertex, we can embed the remaining graph 
induced by $A_v(s_\ell)$
into a binary tree of height $\ell$ (which necessarily has $2(2^{\ell} -1)$ edges).
In particular, each vertex $w \in A_v(s_\ell)$ can be reached by a path of length $\leq \ell$ and therefore
by assumption $w \notin S^N$ and $A_v(s_\ell)$ and $S^N$ are disjoint.
Since for each $u \in S^N$  Proposition~\ref{prop:ext}(iii) implies that $R^+(u) \subseteq S^N$,
we can fix any such $u$ so that $A_v(s_\ell)$ and $R^+(u)$ are disjoint.

Let $\delta \in (0,1)$. From~Lemma~\ref{le:F2}(ii), we can assume that the number of edges pointing towards $R^+(u)$ is bounded from below by  $2Nx^*(1-x^*) (1-\delta)$.
Also the  number of available in-half-edges after the exploration of $R^+(u)$ is trivially less than $2N$.
Therefore, the probability that at a certain stage in the exploration of $R^+(v)$ we connect to an in-half-edge incident to $R^+(u)$
-- provided we have not connected before -- is bounded from below by
\[ 
\frac{2 N x^*(1-x^*) (1-\delta)}{2 N} = x^*(1-x^*)(1-\delta). 
\]
Now,  if $A_v(s_\ell)$ and $R^+(u)$ are disjoint, then we have avoided connecting to $R^+(u)$ for the first $s_\ell$ steps
of the exploration of $R^+(v)$. Therefore, the probability of the latter event is bounded from above by
\begin{equation*}
\label{eq:2702-1} 
( 1  - x^*(1-x^*)(1-\delta))^{s_\ell} 
= ( 1  - x^*(1-x^*)(1-\delta))^{2 (( \log N)^{1+\eps} -1)} =  o\Big(\frac{1}{N^2}\Big). 
\end{equation*}
In particular the probability that $A_v(s_\ell)$ and $R^+(u)$ are disjoint still tends to zero after summing over all possible vertices $u,v$.

(ii)
Let  $s = (1+\eps) \frac{2}{- \log (4 x^*(1-x^*))} \log N$ for some $\eps \in (0,1)$. 
Note that by Proposition~\ref{prop:ext}(iii) with high probability for any $u \in S^N$, we have $R^+(u) \subseteq S^N$.
Hence, 
\[ 
\begin{aligned}  \p \{ \exists v \notin S^N & \, : \, | R^+(v) \cap (S^N)^c| \geq s \} \\
& \leq \p \{ \exists u,v \, : \, v \notin R^+(u) \mbox{ and } |R^+(v) \cap R^+(u)^c | \geq s \}  + o(1). 
\end{aligned}
\]
Now, fix vertices $u,v$ such that $v \notin R^+(u)$.
We construct a tree $\calT$ rooted at $v$ as follows. We match the first out-half-edge
to one of the remaining in-half-edges. If the chosen in-half edge is incident to a vertex that is neither in $R^+(u)$ nor equal to $v$, 
we add it to the tree $\calT$. Then, we find a partner for the second out-half-edge and only add the corresponding vertex to the tree, if it connects to a previously 
unseen vertex outside $R^+(u)$. We continue this exploration in a breadth-first way and keep adding vertices to $\calT$ only if they are neither in $R^+(u)$ nor have already been explored. In particular, the tree 
is constructed in such a way that  $|R^+(v) \setminus R^+(u)|$
is equal to the number $|\calT|$ of vertices in $\calT$. 

Let $\delta > 0$, which we will later choose small depending on $x^*$ and $\eps$.
Denote by $\widehat W^-(R^+(u))$ the in-half-edges incident to  $R^+(u)$ that have not been used after $R^+(u)$ has been explored
and denote by  $|E(R^+(u))|$ the total number of edges used in the exploration of $R^+(u)$.
If we set 
\[ \calE(u) = \Big\{ |\widehat W^-(R^+(u))| \geq (2x^*(1-x^*) - \delta)N , |E(R^+(u))| \geq (2x^*-\delta)N \Big\}, \]
we know from
Lemma~\ref{le:F2} and Proposition~\ref{prop:F1}(i) that 
\begin{equation}
\label{eq:p}
\p\Big( \bigcap_{w } \calE(w) \Big) \ra 1,
\end{equation}
so that we can assume that $\calE(u)$ holds.

Moreover, we again let $A_v(k)$ be the vertices found in the first $k$ steps of the exploration starting from $v$.
Then, the probability that we do not add a vertex to the tree $\calT$ in the $k$th step of the exploration of $R^+(v)$
is bounded from below by
\[ 
\frac{ | \widehat W^-(R^+(u)) | + |\bigcup_{w \in A_v(k-1)}  W^-_w | - (k-1)  }{2N - |E(R^+(u))| - (k-1) }, 
\]
where the $k$ in the denominator is an upper bound on the in-half-edges already connected to
either $R^+(u)$ or $A_v(k-1)$ by step $k-1$ and we recall that $W^-_w$ denotes
the set of in-half-edges of vertex $w$. If $k \leq \sqrt{N}$, this expression is by \eqref{eq:p} bounded from below by 
\[  \frac{ 2x^*(1-x^*) - \delta - \frac{1}{\sqrt{N}}}{ 2 - 2x^* + \delta }
\geq  \frac{ 2x^*(1-x^*) -2 \delta}{ 2 - 2x^* + \delta } =: 1-p, \]
provided $N$ is large. Since $x^* > \frac 12$,  we can assume that  $\delta$ is chosen small enough such that
 $p < \frac 12$.

Let $\widehat \calT$ be a sub-critical Galton-Watson tree, where each individual has offspring distribution given
by the distribution of the sum of two independent Bernoulli variables with success probability $p$ each. 
Thus, by the above estimates and the independence structure of the configuration model, we can couple 
$\calT$ and $\widehat \calT$ such that with probability $1$ the graph induced by the first  $\lfloor \sqrt{N}/2 \rfloor$ 
vertices (counted breadth-first) of $\calT$ is a subgraph of $\widehat \calT$. Hence, by construction we can deduce that since for $N$ large, 
$s \leq \lfloor \sqrt{N}/2 \rfloor$ 
\begin{equation}
\label{eq:0905-1} 
\p \{ |R^+(v) \setminus R^+(u)| \geq s , \calE(u)\} \leq
 \P \{ |\calT| \geq s, \calE(u) \} \leq \p \{ | \widehat \calT | \geq s \} . 
 \end{equation}
However, the distribution of the total number of individuals in a Galton-Watson tree is well-understood. Indeed, the Otter-Dwass formula, see e.g.~\cite[Sec.\ 6.2]{Pitman06}, states that
\begin{equation}\label{eq:2805-1} 
\p \{ | \widehat \calT | = n \} = \frac{1}{n} \p\{ S_n = n - 1\} , 
\end{equation}
where $S_n = \sum_{i=1}^n X_i$ and $(X_i)_{i \geq 1}$ is an i.i.d.\ sequence with the same distribution as the number of offspring in the Galton-Watson
process.  Note that in our case $S_n$ has the  same distribution as the sum of  $2n$ independent Bernoulli variables with parameter $p$. 
Thus, if we denote by 
\[ 
I_p(a) = a \log \Big( \frac{a}{p} \Big) + (1-a) \log \Big( \frac{1-a}{1-p} \Big), 
\]
the large deviation rate function for a Bernoulli variable, Cram\'er's theorem tells us that since $p < \frac 12$,
\[ 
\p \{ S_n = n-1 \} =
\p \Big\{ S_n = 2n\Big(\frac{1-\frac1n}{2}\Big) \Big\}
\leq e^{- 2n I_p( \frac{1}{2} (1 - \frac{1}{n}))}  = e^{- n (-  \log (4 p(1-p)) + o(1))} . 
\]
Thus, we can deduce from~\eqref{eq:0905-1} and~\eqref{eq:2805-1}, 
\begin{align}
\label{eq:0905-2}
\p \{ |R^+(v) \setminus R^+(u)| \geq s , \calE(u)\} 
	&\leq \sum_{n \geq s} \p \{ |\widehat{\cal{T}}| = n\}\notag \\
	&\leq \frac{1}{s} \sum_{n \geq s}  e^{-n (- \log (4 p(1-p)) + o(1))}   
\end{align}
Now, we choose $\delta$ small enough 
 such that $p = p(\delta)$ satisfies
\[  \log 4 p(1-p) \leq \frac{\log 4 x^*(1-x^*) }{1+\frac 12 \eps} . \]
Then, it follows from~\eqref{eq:0905-2} and our choice of $s = (1+\eps) \frac{2}{-\log (4x^*(1-x^*))} \log N$ that
\[ \begin{aligned} \p \{ |R^+(v) \setminus R^+(u)| \geq s , \calE(u)\}   & \leq 2
 e^{-s (- \log (4 p(1-p)) + o(1))} \\ &\leq N^{ - 2 ( \frac{1+\eps}{1+  \eps/2} + o(1)) } 
  \leq N^{- 2( 1 + \frac 14 \eps + o(1))} , \end{aligned}
\]
Finally, since this expression tends to zero even after summing over all possible pairs
of vertices  $u$ and $v$,  claim (ii) follows.

(iii) If $\calC$ is the second largest strongly connected component of $\DCM_N$, then
 $\calC \subseteq (S^N)^c$.
Moreover, 
$\calC \subseteq R^+(v)$ for any 
$v \in \calC$ and thus $\calC \subseteq R^+(v) \setminus S^N$ for some
vertex $v \notin S^N$. Therefore, the required bound on $|\calC|$ follows
from (ii).
\end{proof}

{\bf Acknowledgement.} JB and MO gratefully acknowledge support by DFG Priority Programme 1590 ``Probabilistic Structures in Evolution''.

%
\bibliographystyle{alpha}
\bibliography{cyclical}

\begin{thebibliography}{WKLR12}

\bibitem[Bol80]{B80}
B.~Bollob{\'a}s.
\newblock A probabilistic proof of an asymptotic formula for the number of
  labelled regular graphs.
\newblock {\em European J. Combin.}, 1(4):311--316, 1980.

\bibitem[CF04]{CF04}
C.~Cooper and A.~Frieze.
\newblock The size of the largest strongly connected component of a random
  digraph with a given degree sequence.
\newblock {\em Combin. Probab. Comput.}, 13(3):319--337, 2004.

\bibitem[CF12]{CF12}
C.~Cooper and A.~Frieze.
\newblock Stationary distribution and cover time of random walks on random
  digraphs.
\newblock {\em J. Combin. Theory Ser. B}, 102(2):329--362, 2012.

\bibitem[Cha99]{Chang}
J.~T. Chang.
\newblock Recent common ancestors of all present-day individuals.
\newblock {\em Adv. in Appl. Probab.}, 31(4):1002--1038, 1999.

\bibitem[Hof13]{RemcoNotes}
R.~van~der Hofstad.
\newblock {\em Random Graphs and Complex Networks}.
\newblock 2013.
\newblock Lecture notes in preparation, available at
  \url{http://www.win.tue.nl/~rhofstad/}.

\bibitem[JDP83]{JDP83}
K.~Joag-Dev and F.~Proschan.
\newblock Negative association of random variables, with applications.
\newblock {\em Ann. Statist.}, 11(1):286--295, 1983.

\bibitem[LPW09]{LPW09}
D.~A. Levin, Y.~Peres, and E.~L. Wilmer.
\newblock {\em Markov chains and mixing times}.
\newblock American Mathematical Society, Providence, RI, 2009.

\bibitem[Pit06]{Pitman06}
J.~Pitman.
\newblock {\em Combinatorial stochastic processes}, volume 1875 of {\em Lecture
  Notes in Mathematics}.
\newblock Springer-Verlag, Berlin, 2006.

\bibitem[WKLR12]{WKLR12}
J.~Wakeley, L.~King, B.~S. Low, and S.~Ramachandran.
\newblock Gene genealogies within a fixed pedigree, and the robustness of
  {K}ingman's coalescent.
\newblock {\em Genetics}, 190(4):1433--1445, 2012.

\end{thebibliography}

\end{document}